\documentclass[a4paper,10pt]{amsart}
\usepackage{latexsym, amssymb,amsmath,amsthm}
\usepackage{mathrsfs}
\usepackage{bbm}

\usepackage[all, knot]{xy}
\xyoption{arc}
\def \To{\longrightarrow}
\def \dim{\operatorname{dim}}
\def \Gr{\operatorname{Gr}}
\def \GR{\mathcal{GR}}

\def \mod{\operatorname{mod}}
\def \C{\mathbb{C}}

\def \q{\mathbbm{q}}
\def \c{\mathcal{C}}
\def \D{\Delta}
\def \d{\delta}
\def \e{\varepsilon}

\def \Z{\mathbb{Z}}

\def \k{\mathrm{k}}
\def \1{\mathbf{1}}
\def \Id{\operatorname{Id}}

\numberwithin{equation}{section}

\newtheorem{theorem}{Theorem}[section]
\newtheorem{lemma}[theorem]{Lemma}
\newtheorem{proposition}[theorem]{Proposition}
\newtheorem{corollary}[theorem]{Corollary}

\newtheorem{remark}[theorem]{Remark}

\begin{document}

\title{THE GREEN RINGS OF POINTED TENSOR CATEGORIES OF FINITE TYPE}



\author[H.-L. Huang]{Hua-Lin Huang}
\address{School of Mathematics, Shandong University,
Jinan 250100, China} \email{hualin@sdu.edu.cn}

\author[F. Van Oystaeyen]{Fred Van Oystaeyen}
\address{Department of Mathematics and Computer Science, University of Antwerp, Middelheimlaan
1, B-2020 Antwerp, Belgium} \email{fred.vanoystaeyen@ua.ac.be}

\author[Y. Yang]{Yuping Yang}
\address{School of Mathematics, Shandong University,
Jinan 250100, China} \email{yuping.yang@mail.sdu.edu.cn}

\author[Y. Zhang]{Yinhuo zhang}
\address{Department WNI, University of Hasselt, Universitaire Campus, 3590 Diepeenbeek, Belgium}
\email{yinhuo.zhang@uhasselt.be}

\date{}
\maketitle

\begin{abstract}
In this paer, we compute the Clebsch-Gordan formulae and the Green rings of connected pointed tensor categories of finite type.
\end{abstract}

\section{Introduction}
Throughout the paper, $\k$ is an algebraically closed field of characteristic
zero, and (co)algebras, (co)modules, categories, etc are over $\k.$
Concepts and notations about tensor categories are adopted from \cite{egno}.

Let $\c$ be a tensor category, that is, a locally finite abelian
rigid monoidal category in which the neutral object is simple. Note
that the Jordan-H\"older theorem and the Krull-Schmidt theorem hold in $\c.$
For each $M \in \c,$ let $[M]$ denote the associated iso-class. The
Green ring $\GR(\c)$ of $\c$ is the ring with generators the
iso-classes $[M]$ of $\c,$ and relations \[ [M]+[N]=[M \oplus N],
\quad [M] \cdot [N] = [M \otimes N]. \] By the Krull-Schmidt
theorem, the ring $\GR(\c)$ is a free $\Z$-module, and the set of
indecomposable iso-classes of $\c$ forms a basis. When $\c=H$-$\mod,$ where $H$ is a Hopf algebra, the ring $\GR(\c)$ is also called the Green ring of $H.$

The Green ring $\GR(\c)$ provides a convenient way of organizing
information about direct sums and tensor products of $\c.$ When $\c$
is the category of modular representations of a finite group, the
ring $\GR(\c)$ was investigated systematically for the first time by
J. A. Green \cite{gr}, hence the notion Green ring is widely used in
literature. Green rings have played an important role in the modular
representation theory, see \cite{b} and references therein.

Without a doubt, Green rings should be equally important in the
study of those tensor categories of (co)modules over (co)quasi-Hopf
algebras (or quasi-quantum groups) which are not semisimple, see
\cite{egno,eo}. However, so far there are not many results obtained
about the Green rings of such tensor categories due to the obvious
complexity. In \cite{w1}, the Green ring (called representation ring
there) of the quantum double of a finite group was described by a
direct sum decomposition. In a subsequent work \cite{w2}, the Green
ring of the twisted quantum double of a finite group was considered
and some results about the Grothendieck ring, that is the quotient
of the Green ring by the ideal of short exact sequences, were
obtained. Recently, the Green rings of the Taft algebras and the generalized Taft algebras were
presented by generators and relations in \cite{coz} and \cite{lz} respectively.

The aim of this paper is to compute the Clebsch-Gordan formulae and to present
the Green rings of pointed tensor categories of finite type as
classified in \cite{qha3}. Such tensor categories are actually the
comodule categories of some pointed coquasi-Hopf algebras of finite
corepresentation type, or equivalently the module categories of some elementary quasi-Hopf algebras of finite representation type, in which there are only finitely many
iso-classes of indecomposable objects, and as such their Green rings
are suitable to study. Among them, the module categories of the Taft and generalized Taft algebras are particular examples. In addition, these tensor categories can be
presented by quiver representations and we can take full advantage
of the handy quiver techniques (see e.g. \cite{ass,w}) for the
computations. In Section 2, we recall some necessary facts about
pointed tensor categories of finite type. The Clebsch-Gordan
formulae for such tensor categories are computed in Section 3.
Finally in Section 4, the Green rings are determined. The results
obtained here extend those in \cite{cc,coz,lz} to a much greater range in a unified way.

\section{Pointed tensor categories of finite type}

In this section, we recall pointed tensor categories of finite type
and their presentations via quiver representations.

\subsection{}
Let $n>1$ be a positive integer and $C_n=<g|g^n=1>$ the cyclic group
of order $n.$ By $Z_n$ we denote the cyclic quiver with vertices
indexed by $C_n$ and with arrows $a_i: g^i \to g^{i+1}$ where the
indices $i$ are understood as integers modulo $n.$ Let $p_i^l$
denote the path $a_{i+l-1} \cdots a_{i+1}a_i$ of length $l.$ The
$\k$-span of $\{p_i^l \ | \ 0 \le i \le n-1, \ l \ge 0 \}$ is called
the associated path space of the quiver $Z_n,$ and denoted by $\k
Z_n.$

From now on, let $m'$ denote the remainder of the division of the
integer $m$ by $n.$ There is a natural coalgebra structure on $\k
Z_n$ with coproduct and counit given by \begin{equation}
\D(p_i^l)=\sum_{m=0}^l p_{(i+m)'}^{l-m} \otimes p_i^m, \quad
\e(p_i^l)=\d_{l,0}=\left\{
                     \begin{array}{ll}
                       1, & \hbox{$l=0$;} \\
                       0, & \hbox{otherwise.}
                     \end{array}
                   \right.
\end{equation} This is the so-called path coalgebra of the quiver
$Z_n.$

\subsection{}
The quiver $Z_n$ is in fact a Hopf quiver in the sense of \cite{cr}.
By \cite{qha1}, there exist  graded coquasi-Hopf algebra, also known
as Majid algebra, structures on the path coalgebra $\k Z_n.$
Moreover, the graded coquasi-Hopf structures are parameterized by
some set of 1-dimensional projective representation of $C_n,$ see
\cite{qha2} for a general theorem.

Explicit graded coquasi-Hopf structures on $\k Z_n$ are given in
\cite{qha3}. To state this result, firstly we need to fix some
notations. For any $\hbar \in \k,$ define $m_\hbar=1+\hbar+\cdots
+\hbar^{m-1}$ and $m!_\hbar=1_\hbar \cdots m_\hbar.$ The Gaussian
binomial coefficient is defined by
$\binom{m+n}{m}_\hbar:=\frac{(m+n)!_\hbar}{m!_\hbar n!_\hbar}.$ Let
$\q$ be a primitive $n$-th root of unity. Assume $0 \le s \le n-1$
is an integer and $q$ an $n$-th root of $\q^s.$ For each pair
$(s,q),$ there is a unique graded coquasi-Hopf algebra $\k Z_n(s,q)$
on $\k Z_n$ with multiplication given by
\begin{equation}
p_i^l\cdot
p_j^m=\mathbbm{q}^{-sjl}q^{-jl}\mathbbm{q}^{\frac{s(i+l')[m+j-(m+j)']}{n}}{l+m\choose
l}_{\mathbbm{q}^{-s}q^{-1}}p_{(i+j)'}^{l+m} \ .
\end{equation}
Let $d$ be the order of $q.$ Clearly, $d|n$ if $s=0$ and
$d=\frac{n^2}{(s,n^2)}$ if $1\leq s\leq n-1.$ Let $\k Z_n(d)$ denote
the subcoalgebra of $\k Z_n$ which as a $\k$-space is spanned by
$\{p_i^l \ | \ 0 \le i \le n-1, \ 0 \le l \le d-1 \}.$ The
multiplication (2.1) is closed inside $\k Z_n(d)$ and it becomes a
graded coquasi-Hopf algebra, denoted by $M(n,s,q).$

\subsection{}
A pointed tensor category of finite type is a tensor category in
which there are only finitely many iso-classes of indecomposable
objects and whose simple objects are invertible, see \cite{egno,
qha3}.

By $\c(n,s,q)$ we denote the category of finite-dimensional right
$M(n,s,q)$-comodules. It is proved in \cite{qha3} that any pointed
tensor category of finite type consists of finitely many identical
components, and the connected component containing the neutral
object is equivalent to a deformation of some $\c(n,s,q),$ and See
\cite[Theorem 4.2 and Corollary 4.4]{qha3} for a full description.

In this paper we focus on the Clebsch-Gordan formulae and the Green
ring of $\c(n,s,q).$ The results can be easily extended to the
non-connected case. Now we give an explicit presentation of
$\c(n,s,q)$ by quiver representations. Recall that a representation
of the quiver $Z_n$ is a collection $V=(V_i,T_i)_{0 \le i \le n-1}$
where $V_i$ is a vector space corresponding to the vertex $g^i$ and
$T_i: V_i \to V_{i+1}$ is a linear map corresponding to the arrow
$a_i.$ Given a path $p_i^l,$ we define a corresponding linear map
$T_i^l$ as follows. If $l=0,$ then put $T_i^0=\Id_{V_i}.$ If $l>0,$
put $T_i^l=T_{i+l-1} \cdots T_{i+1}T_i.$ The category $\c(n,s,q)$
consists of the representations $V$ of $Z_n$ such that $T_i^l=0$
whenever $l \ge d.$

The indecomposable objects of $\c(n,s,q)$ can be described as
follows. Assume $0 \le i \le n-1$ and $0 \le e \le d-1.$ Let
$V(i,e)$ be a vector space of dimension $e+1$ with a basis
$\{v_m^i\}_{0 \leq m \leq e}.$  $V(i,e)$ is made into a
representation of $Z_n$ by putting $V(i,e)_j$ the $\k$-span of
$\{v_m^i | (i+m)'=j\}$ and letting $T(i,e)_j$ maps $v_m^i$ to
$v^i_{m+1}$ if $(i+m)'=j.$ Clearly $V(i,e)$ is an object in
$\c(n,s,q)$ and $\{V(i,e)\}_{0\leq i\leq n-1, \ 0\leq e\leq d-1}$ is
a complete set of iso-classes of indecomposable objects of
$\c(n,s,q).$

\subsection{}
For later computations, we need to know how the quiver
representation $V(i,e)$ is viewed as a right $M(n,s,q)$-comodule.
Let \[\delta :  V(i,e) \to V(i,e) \otimes M(n,s,q)\] denote the
comodule structure map, then \begin{equation}
\d(v_m^i)=\sum_{x=m}^{e} v_x^i\otimes
p_{(i+m)'}^{x-m}.\end{equation}

For any two indecomposable objects $V(i,e)$ and $V(j,m)$ of
$\c(n,s,q),$ the comodule structure map of their tensor product
$V(i,e) \otimes V(j,m)$ is given by
\begin{equation}
\d(v_a^i\otimes v_b^j)=\sum_{x=s}^e\sum_{y=t}^m
C_{(i,a)(j,b)}^{(i,x)(j,y)} v_x^i\otimes v_y^j\otimes
p_{(i+a+j+b)'}^{x+y-a-b},
\end{equation}
where $C_{(i,a)(j,b)}^{(i,x)(j,y)}$ is the coefficient of the
product $p_{(i+a)'}^{x-a} \cdot p_{(j+b)'}^{y-b}$ as given in (2.2).
It is clear that $C_{(i,a)(j,b)}^{(i,x)(j,y)}=0$ whenever $x+y-a-b
\geq d.$

\section{The Clebsch-Gordan formulae}

In this section we compute the complete decomposition of $V(i,e) \otimes
V(j,f)$ with the help of some techniques from quivers and their representations.

\subsection{Reduction of $V(i,e)\otimes V(j,f)$ to $V(0,e)\otimes V(0,f)$}

First of all, we list some useful facts of indecomposable comodules of
$M(n,s,q)$.

\begin{lemma}
\begin{gather}
V(0,0)\otimes V(i,e) \cong V(i,e) = V(i,e)\otimes V(0,0),\\
V(1,0)^{\otimes l} = V(l',0), \quad V(1,0)^{\otimes n} = V(0,0), \\
V(1,0)\otimes V(i,e) = V((i+1)',e) = V(i,e)\otimes V(1,0), \\ \forall \ 0\leq i\leq n-1, \ 0\leq e\leq d-1. \notag
\end{gather}
\end{lemma}

\begin{proof}
(3.1) and (3.2) are easy by definition. We prove only (3.3). Let $\{v_0^1\}$ and $\{v_k^i\}_{0\leq k\leq e}$ be respectively a basis of $V(1,0)$
and $V(i,e).$ Then $\{v_0^1\otimes v_k^i\}_{0\leq k\leq
e}$ is a basis of $V(1,0)\otimes V(i,e)$ and the comodule
structure is given by
\begin{equation}\d(v_0^1\otimes v_k^i)=\sum_{x=k}^e C_x
v_0^1\otimes v_x^i\otimes p_{(i+k+1)'}^{x-k}.
\end{equation}
Note that $C_x \neq 0$ by (2.4). Now view $V(1,0)\otimes V(i,e)$ as a representation of $Z_n.$ (2.4) implies that $v_0^1\otimes v_k^i$ should lie in the vector space attached to the vertex $g^{(i+k+1)'}.$ On the other hand, by (3.2) $V(1,0)$ is invertible. It follows that $ V(1,0)\otimes V(i,e)$ is indecomposable. Now by comparing it with the complete set of iso-classes of indecomposable comodules of $M(n,s,q),$ it is easy to see that $V(1,0)\otimes V(i,e) = V((i+1)',e).$ Similarly, we have $V(i,e) \otimes V(1,0) = V((i+1)',e).$
\end{proof}

Thanks to the lemma, we can reduce the decomposition of $V(i,e)\otimes V(j,f)$ to $V(0,e)\otimes V(0,f).$

\begin{corollary}
\begin{gather}
V(i,e) = V(1,0)^{\otimes i} \otimes V(0,e) = V(0,e) \otimes V(1,0)^{\otimes i}, \\
V(i,e)\otimes V(j,f) = V(1,0)^{\otimes (i+j)} \otimes V(0,e)\otimes V(0,f).
\end{gather}
\end{corollary}

\begin{proof}
Follows easily from Lemma 3.1.
\end{proof}

\subsection{Decomposition of $V(0,e)\otimes V(0,f)$ at vertices}

In the rest of this section, we write $V=V(0,e)\otimes V(0,f)$ for brevity. View $V$ as a representation of $Z_n$ and let $V_i$ denote the corresponding subspace attached to the vertex $g^i.$ Note that $V=\oplus_{i=0}^{n-1} V_i$ as a vector space. This is the so-called decomposition at vertices, providing essential information for us to detect its direct summands in the next subsection.

With the same notations as in Subsection 2.3, we have a basis $\{v_k^0\otimes v_l^0\}_{0\leq k\leq e,0\leq l\leq f}$ of $V.$ By (2.4), it is clear that
\begin{equation} V_i=\C\{v_k^0\otimes v_l^0
|(k+l)'=i, \ 0\leq k\leq e, \ 0\leq l\leq f \}
\end{equation} and hence its dimension is given by
\begin{equation}
\dim V_i=|\{(k,l)\in \Z\times\Z:(k+l)'=i,\ 0\leq k\leq e,\ 0\leq l\leq
f\}|.
\end{equation}

In order to give an explicit formula for $\dim V_i,$ we need the following technical lemma.
\begin{lemma}
For  integers $0\leq \alpha , \beta \leq n-1$, let $$N_{(\alpha,
\beta)}^k=|\{(x,y)\in \Z\times \Z:(x+y)'=k,0\leq x\leq \alpha, 0\leq
y\leq \beta\}|.$$ If $\alpha \leq \beta$ and $\alpha + \beta \leq n-1,$ then
\begin{equation} N_{(\alpha,\beta)}^k=
\begin{cases} k+1, & 0\leq k\leq \alpha,  \\ \alpha + 1, &\alpha \leq
k \leq \beta, \\ \alpha + \beta +1 -k, & \beta \leq k \leq \alpha
+\beta;
\end{cases}
\end{equation}
if $\alpha \leq \beta$ and $\alpha + \beta \geq n,$ set $\gamma = \alpha + \beta -n+1,$ then
\begin{equation}
N_{(\alpha,\beta)}^k=
\begin{cases}\gamma +1, & 0\leq k\leq \gamma, \\ k+1,  & \gamma \leq k
\leq \alpha, \\ \alpha +1, & \alpha \leq k\leq \beta, \\ \alpha +\beta
+1-k, & \beta \leq k \leq n-1;
\end{cases}
\end{equation}
if $\alpha > \beta$ and $\alpha + \beta \leq n-1,$ then
\begin{equation}
N_{(\alpha,\beta)}^k=
\begin{cases} k+1, & 0\leq k\leq \beta,  \\ \beta + 1, &\beta \leq
k \leq \alpha, \\ \beta + \alpha +1 -k, & \alpha \leq k \leq \alpha
+\beta;
\end{cases}
\end{equation}
if $\alpha > \beta$ and $\alpha + \beta \geq n,$ set $\gamma$ as above, then
\begin{equation}
N_{(\alpha,\beta)}^k=
\begin{cases}\gamma +1, & 0\leq k\leq \gamma, \\ k+1,  & \gamma \leq k
\leq \beta, \\ \beta +1, & \beta \leq k\leq \alpha, \\ \alpha +\beta
+1-k, & \alpha \leq k \leq n-1.
\end{cases}
\end{equation}
\end{lemma}

\begin{proof}
By direct computation.
\end{proof}

\begin{lemma}
\begin{equation}
 \dim V_i=\frac{ef-e'f'}{n}+\frac{e+f-e'-f'}{n}+N_{(e',f')}^i.
\end{equation}
\end{lemma}

\begin{proof}
Write $e=\alpha n + e',$ $f =\beta n+ f'$ and $$\{(x,y):0\leq
x\leq e,0\leq y\leq f\}=U_1\cup U_2\cup U_3 \cup U_4$$ where
\begin{eqnarray*}
U_1&=&\{(x,y):0\leq x\leq \alpha n-1, 0\leq y\leq \beta n-1\}, \\
U_2&=&\{(x,y):\alpha n\leq x\leq e, 0\leq y\leq \beta n-1\}, \\
U_3&=&\{(x,y):0\leq x\leq \alpha n-1, \beta n\leq y\leq f\}, \\
U_4&=&\{(x,y):\alpha n\leq x\leq e, \beta n\leq y\leq f\}.\\
\end{eqnarray*}
By easy computation and Lemma 3.3 one has:
\begin{eqnarray*}
|\{(x,y)\in U_1:(x+y)'=i\}|&=&n\alpha\beta, \\
|\{(x,y)\in U_2:(x+y)'=i\}|&=&\beta(e'+1), \\
|\{(x,y)\in U_3:(x+y)'=i\}|&=&\alpha(f'+1), \\
|\{(x,y)\in U_4:(x+y)'=i\}|&=&N_{(e',f')}^i.
\end{eqnarray*}
Now we obtain:
\begin{eqnarray*}
\dim V_i &=& n\alpha\beta + \alpha f' +e'\beta +\alpha +\beta +N_{(e',f')}^i \\ &=& \frac{ef-e'f'}{n}+\frac{e+f-e'-f'}{n}+N_{(e',f')}^i.
\end{eqnarray*}
\end{proof}

\subsection{The Clebsh-Gordan formulae}
The main objective of this subsection is to compute the decomposition of $V=V(0,e)\otimes V(0,f)$ which will easily lead to the general Clebsh-Gordan formulae for $\c=\c(n,s,q).$

By $S_i(U)$ we denote the number of direct summands of the form $V(i,?)$ that appear in the decomposition of $U \in \c.$ As in Subsection 2.3, we view $U$ as a representation of the quiver $Z_n$ and let $T_i :U_i\rightarrow U_{i+1}$ be the linear map attached to the arrow $a_i.$ Write
\begin{equation*}
U_i'=\ker T_i=\{x\in U_i: T_i(x)=0\}.
\end{equation*}

\begin{lemma}
\begin{equation}
S_i(U)=\dim U_i-\dim U_{(i-1)'}+\dim U'_{(i-1)'}.
\end{equation}
\end{lemma}

\begin{proof}
If $U=V(i,l),$ then one observes that \[\dim U_i-\dim U_{(i-1)'}+\dim U'_{(i-1)'}=1=S_i(U);\] if $U=V(m,l)$ with $m \neq i,$ one has \[\dim U_i-\dim U_{(i-1)'}+\dim U'_{(i-1)'}=0=S_i(U).\]

In general, if $U=\oplus V(j,l),$ then $U_i= \oplus V(j,l)_i$ and $U'_i=\oplus V(j,l)'_i.$ Now we have
\begin{eqnarray*}
S_i(U)&=&\sum S_i(V(j,l))\\
&=&\sum [\dim V(j,l)_i-\dim V(j,l)_{(i-1)'}+\dim V(j,l)'_i]\\
&=&\dim U_i-\dim U_{(i-1)'}+\dim U'_{(i-1)'}.
\end{eqnarray*}
\end{proof}

Now we come back to $V=V(0,e)\otimes V(0,f).$ To compute $S_i(V),$ it suffices to determine $\dim V'_i.$

\begin{proposition}
\begin{equation}\dim V'_i= \left\{
               \begin{array}{ll}
                 \mid \{j: f-e \leq j \leq f, (e+j)'=i\}\mid, & \hbox{if $e \le f$;} \\
                 \mid \{j: e-f \leq j \leq e, (e+j)'=i\}\mid, & \hbox{if $e >f.$}
               \end{array}
             \right.
\end{equation}
\end{proposition}
\begin{proof}
Clearly the two cases are symmetric, so it is enough to prove one of them. Assume $e \le f.$ Take the bases $\{v^0_l\}_{0 \le l \le e}$ and $\{v^0_m\}_{0 \le m \le f}$ for $V(0,e)$ and $V(0,f)$ respectively as before. Assume $u \in V'_i$. We can write $$u=\sum_{(l+m)'=i} \lambda_{lm} v_l^0\otimes v_m^0=\sum_x \sum_{l+m=nx+i} \lambda_{lm} v_l^0\otimes v_m^0.$$
Let $u_x=\sum_{l+m=nx+i} \lambda_{lm} v_l^0\otimes v_m^0.$ Then it is clear that $T(u)=0$ if and only if $T(u_x)=0, \forall x.$ Rewrite $u_x=\lambda_1 v^0_{l_1} \otimes v^0_{m_1} + \cdots + \lambda_y v^0_{l_y} \otimes v^0_{m_y}$ with $\lambda_k \ne 0, \forall k$, such that $l_1 > \cdots > l_y$ and hence $m_1 < \cdots < m_y.$ Recall that:
 $$T_i(v_x^0\otimes v_y^0)=\mathbbm{q}^{-sy}q^{-y} v_{x+1}^0\otimes v_y^0 + \mathbbm{q}^{\frac{sx[1+y-(1+y)']}{n}}v_x^0\otimes v_{y+1}^0 $$ if $x<e, y<f.$ It follows that if $l_1\neq e,$ we have: $$T_i(u_x)=\lambda_1\mathbbm{q}^{-sm_1}q^{-m_1} v_{l_1+1}^0\otimes v_{m_1}^0 + \sum \mathrm{other \ terms} \neq 0.$$
This forces $l_1=e$ if $T_i(u_x)=0.$ By a similar argument one has $m_y=f$ if $T_i(u_x)=0.$ Clearly it is necessary that $l_1 \le f$ and the subindex of the first factor of the last term $e+l_1-f \ge 0.$ That is, $f-e \le l_1 \le f.$

By appropriate rescaling $u_x$ is of the following form $$u_x=v_e^0\otimes v_j^0 + a_1v_{e-1}^0\otimes v_{j+1}^0 +\cdots +a_{(f-j)} v_{e+j-f}^0\otimes v_f^0$$ with $j=nx+i-e$ and $f-e \leq j \leq f.$ By imposing the condition $T_i(u_x)=0$ one can solve out a unique set of the $a_i$'s by recursion. As the explicit solution will not be used in later argument, we omit the calculation.

Now we have proved that $V'_i$ is spanned by elements like $u_x.$ Clearly the $u_x$'s are linearly independent. Therefore, the dimension of $V'_i$ is equal to the number of such $u_x$'s, which is exactly $\mid \{j: f-e \leq j \leq f, (e+j)'=i\}\mid.$
\end{proof}

With this preparation, now we can compute $S_i=S_i(V).$

\begin{proposition}
Suppose $e=\alpha n+e'$, $f=\beta n+f'$. Then we have:
\begin{equation} S_i=
\begin{cases} \alpha +1, & e \le f, \ 0\leq i\leq e'; \\ \alpha, & e \le f, \ e'<i\leq n-1; \\ \beta +1, & e >f, \ 0\leq i\leq f'; \\ \beta, & e >f, \ f'<i\leq n-1.
\end{cases}
\end{equation}
\end{proposition}
\begin{proof}
We prove only the cases with $e \leq f.$ Write $d^i$ for $\dim V'_i.$ Assume first $e'+f'\leq n-1$ and $e'\leq f'$.

If $e'+f' < n-1$, by applying Lemmas 3.3-3.5 we have:
\begin{equation*}S_i=
\begin{cases} d^{i-1}+1, & 0\leq i\leq e'; \\d^{i-1}, &
e'+1\leq i\leq f'; \\d^{i-1}-1, & f'+1\leq i\leq e'+f'+1;
\\d^{i-1}, & e'+f'+2\leq i\leq n-1.
\end{cases}
\end{equation*}
Further application of Proposition 3.6 leads to:
\begin{equation*}S_i=
\begin{cases} \alpha +1, & 0\leq i\leq e'; \\ \alpha,  &
e'+1\leq i\leq f'; \\ \alpha, & f'+1\leq i\leq e'+f'+1;
\\ \alpha, & e'+f'+2\leq i\leq n-1.
\end{cases}
\end{equation*}
Note that if $e'+f'+1=n-1$, the last case actually vanishes.

If $e'+f'=n-1$, then by Lemmas 3.3-3.5 we have:
\begin{equation*}S_i=
\begin{cases} d^{i-1}+1, & i=0; \\d^{i-1}, &
1\leq i\leq e'; \\d^{i-1}-1, & e'+1\leq i\leq f';
\\d^{i-1}, & f'+1\leq i\leq e'+f'=n-1.
\end{cases}
\end{equation*}
Following Proposition 3.6 we have:
\begin{equation*}S_i=
\begin{cases} \alpha +1, & i=0; \\ \alpha +1, & 1\leq i\leq e'; \\ \alpha, & e'+1\leq i\leq f';
\\ \alpha,  & f'+1\leq i\leq n-1.
\end{cases}
\end{equation*}
Therefore,
\begin{equation*}S_i=
\begin{cases} \alpha +1, & 0\leq i\leq e'; \\ \alpha,  & e'+1\leq i\leq n-1.
\end{cases}
\end{equation*}

The remaining three cases with $e'+f'\leq n-1, e'>
f'; e'+f'\geq n, e'\leq f'$ and $e'+f'\geq n, e'>f'$ respectively can be dealt with by the same process, so the detail is omitted.
\end{proof}

From now on, we shall understand $V(i,?)$ as $V(i',?)$ if $i \ge n.$ With this notation, we have the following nice form for the complete decomposition of $V=V(0,e) \otimes V(0,f).$

\begin{corollary}
\begin{equation}V=
\left\{
  \begin{array}{ll}
    V(0,?)\oplus V(1,?)\oplus \cdot \cdot \cdot \oplus V(e,?), & \hbox{if $e \le f;$} \\
    V(0,?)\oplus V(1,?)\oplus \cdot \cdot \cdot \oplus V(f,?), & \hbox{otherwise.}
  \end{array}
\right.
\end{equation}
\end{corollary}

\begin{proof}
Direct consequence of Proposition 3.7.
\end{proof}

Now we are in the position to give the main result of this section.

\begin{theorem}
If $e+f\leq d-1,$ then
\begin{equation} V=
\left\{
  \begin{array}{ll}
   \bigoplus_{i=0}^e V(i,e+f-2i), & \hbox{if $e \le f$;} \\
   \bigoplus_{i=0}^f V(i,e+f-2i), & \hbox{otherwise.}
  \end{array}
\right.
\end{equation}

If $e+f \geq d,$ set $\gamma =e+f-d+1,$ then
\begin{equation}V=
\left\{
  \begin{array}{ll}
    \bigoplus_{i=0}^\gamma V(i,d-1) \oplus \bigoplus_{j=\gamma + 1}^{e} V(j,e+f-2j), & \hbox{if $e \le f$;} \\
    \bigoplus_{i=0}^\gamma V(i,d-1) \oplus \bigoplus_{j=\gamma + 1}^{f} V(j,e+f-2j), & \hbox{otherwise.}
  \end{array}
\right.
\end{equation}
\end{theorem}

\begin{proof}
As before, it is enough to show the case with $e \leq f.$ We prove by induction on $e.$ First of all we consider the case with $e+f \le d-1.$ When $e=0,$ the claim is trivial. Consider $e=1.$ By Corollary 3.8, we know $V(0,1) \otimes V(0,f) = V(0,l) \oplus V(1,m)$ for some $l,m.$

Now we view $V$ as a quiver representation and keep the same notations as in Section 2. Note that, under our assumptions, $f < d-1,$ $T_1^f(V) \neq 0$ and $T_1^{f+1}(V)=0.$ It follows that $m \leq f$ and so $T_0^f(V(1,m))=0.$ On the other hand, $T_0^{f+1}(V) \neq 0$ and $T_0^{f+2}(V)=0$ imply that $T_0^{f+1}(V(0,l))\neq 0$ and $T_0^{f+2}(V(0,l))=0.$ That is to say, $l=f+1.$ By comparing the dimensions, we get $m=f-1.$
Thus we have:
\begin{equation}V(0,1)\otimes V(0,f)=V(0,f+1)\oplus V(1,f-1).
\end{equation}
Similarly, one can prove that:
\begin{equation}
V(0,f)\otimes V(0,1)=V(0,f+1)\oplus V(1,f-1).
\end{equation} In particular, we obtain that:
\begin{equation} V(0,1)\otimes V(0,f)=V(0,f)\otimes V(0,1). \end{equation}

Now we assume $1<e \leq f$ and $e+f\leq d-1.$ By (3.21) we have
\begin{equation}
 [V(0,e)\oplus V(1,e-2)]\otimes V(0,f)=V(0,e-1)\otimes V(0,1)\otimes V(0,f).
\end{equation}
By the induction hypothesis, we have
\begin{eqnarray*}
V(1,e-2)\otimes V(0,f)
&=&V(1,0)\otimes V(0,e-2)\otimes V(0,f)\\
&=&\bigoplus_{i=0}^{e-2} V(i+1,e-2+f-2i),
\end{eqnarray*} and
\begin{eqnarray*}
\lefteqn{V(0,e-1)\otimes V(0,1)\otimes V(0,f)}\\
&=&V(0,e-1)\otimes [V(0,f+1)\oplus V(1,f-1)]\\
&=&[V(0,e-1)\otimes V(0,f+1)] \oplus [(V(1,0)\otimes V(0,e-1)\otimes V(0,f-1)]\\
&=& \bigoplus_{i=0}^{e-1} V(i,e+f-2i) \oplus \bigoplus_{j=0}^{e-1} V(j+1,e+f-2-2j).
\end{eqnarray*}
By subtracting the foregoing two identities, we obtain:
$$V(0,e)\otimes V(0,f)=\bigoplus_{i=0}^e V(i,e+f-2i).$$

Next we prove the case with $e+f \geq d.$  We consider first the case with $f=d-1.$ By Corollary 3.8 and by comparing the dimensions and the numbers of direct summands one can easily obtain:
\begin{equation}
V(0,e) \otimes V(0,d-1)= V(0,d-1)\otimes V(0,e)=\bigoplus_{i=0}^e V(i,d-1).
\end{equation}
In the rest of the proof, we can always assume $e \leq f <d-1,$ hence $e \le d-2.$

For the general situation, we split it into three cases. Namely, $e+f=d, \ e+f=d+1,$ and $e+f> d+1.$

If $e+f=d,$ by the induction hypothesis we have:
\begin{eqnarray*}
V(1,e-2)\otimes V(0,f)&=&V(1,0)\otimes V(0,e-2)\otimes V(0,f) \\
&=& \bigoplus_{i=0}^{e-2}V(i+1,e+f-2-2i)
\end{eqnarray*} and
\begin{eqnarray*}
&&V(0,e-1)\otimes V(0,1)\otimes V(0,f)\\
&=& [V(0,e-1)\otimes V(0,f+1)] \oplus [V(0,e-1)\otimes V(1,f-1)]\\
&=& \bigoplus_{i=0}^1V(i,d-1)\oplus \bigoplus_{j=2}^{e-1}V(j,e+f-2j) \oplus \bigoplus_{k=0}^{e-1}V(k+1,e+f-2-2k).
\end{eqnarray*}
In view of (3.23), we obtain the following by subtracting the above two identities:
\[V(0,e)\otimes V(0,f)=\bigoplus_{i=0}^1V(i,d-1)\oplus \bigoplus_{j=2}^{e}V(j,e+f-2j).\]

If $e+f=d+1$, by induction assumption we have $$V(1,e-2)\otimes V(0,f)=V(1,d-1)\oplus \bigoplus_{i=1}^{e-2}V(i+1,e+f-2-2i)$$ and
\begin{eqnarray*}
&&V(0,e-1)\otimes V(0,1)\otimes V(0,f)\\
&=& [V(0,e-1)\otimes V(0,f+1)]\oplus [V(0,e-1)\otimes V(1,f-1)]\\
&=& \bigoplus_{i=0}^2V(i,d-1)\oplus \bigoplus_{j=3}^{e-1} V(j,e+f-2j)\\
&& \oplus V(1,d-1)\oplus \bigoplus_{k=1}^{e-1}V(k+1,e+f-2-2k).
\end{eqnarray*}
Again subtracting the foregoing two identities gives us:
$$V(0,e)\otimes V(0,f)= \bigoplus_{i=0}^2V(i,d-1)\oplus  \bigoplus_{j=3}^e V(j,e+f-2j).$$

If $e+f > d+1$, then $\gamma =e+f-d+1>2$ and by induction hypothesis we have:

\begin{eqnarray*}
V(1,e-2)\otimes V(0,f)= \bigoplus_{i=0}^{\gamma-2}V(i+1,d-1)\oplus \bigoplus_{j=\gamma-1}^{e-2}V(j,e+f-2-2j).
\end{eqnarray*} and
\begin{eqnarray*}
&&V(0,e-1)\otimes V(0,1)\otimes V(0,f)\\
&=& [V(0,e-1)\otimes V(0,f+1)]\oplus [V(0,e-1)\otimes V(1,f-1)]\\
&=& \bigoplus_{i=0}^{\gamma}V(i,d-1)\oplus \bigoplus_{j=\gamma+1}^{e-1} V(j,e+f-2j)\\
&&  \oplus \bigoplus_{k=0}^{\gamma-2}V(k+1,d-1) \oplus \bigoplus_{l=\gamma-1}^{e-1}V(l+1,e+f-2-2l).
\end{eqnarray*}
By subtracting the above two identities, we obtain:
\begin{eqnarray*}V(0,e)\otimes V(0,f)= \bigoplus_{i=0}^\gamma V(i,d-1) \oplus \bigoplus_{j=\gamma + 1}^{e} V(j,e+f-2j).
\end{eqnarray*}

Thus the proof is completed.
\end{proof}

Now by combining Lemma 3.1 and Theorem 3.9, we arrive at the general Clebsh-Gordan formulae of $\c.$

\begin{theorem}
Let $V=V(i,e)\otimes V(j,f).$ If $e+f\leq d-1,$ then
\begin{equation}V=
\left\{
  \begin{array}{ll}
   \bigoplus_{k=0}^e V(i+j+k,e+f-2k), & \hbox{if $e \le f$;} \\
   \bigoplus_{k=0}^f V(i+j+k,e+f-2k), & \hbox{otherwise.}
  \end{array}
\right.
\end{equation}

If $e+f \geq d,$ set $\gamma =e+f-d+1,$ then

\begin{equation} V=
\left\{
  \begin{array}{ll}
    \bigoplus_{k=0}^\gamma V(i+j+k,d-1) \oplus \bigoplus_{l=\gamma + 1}^{e} V(i+j+l,e+f-2l), & \hbox{if $e \le f$;} \\
    \bigoplus_{k=0}^\gamma V(i+j+k,d-1) \oplus \bigoplus_{l=\gamma + 1}^{f} V(i+j+l,e+f-2l), & \hbox{otherwise.}
  \end{array}
\right.
\end{equation}
\end{theorem}

\begin{remark}
As an immediate consequence of Theorem 3.9, we obtain the commutativity of the tensor product in $\c:$
\begin{equation}
V(i,e)\otimes V(j,f)=V(j,f)\otimes V(i,e).
\end{equation}
\end{remark}

\section{The Green rings}
Denote by $\GR(\c)$ the Green ring of $\c=\c(n,s,q).$  So far we know that $\{[V(i,l)]: 0\leq i \leq n-1,0\leq l\leq d-1\}$ form a basis of   $\GR(\c)$ and Theorem 3.9 provides the ring structure with explicit structure constants. The aim of this section is to present the Green ring $\GR(\c)$ by generators  and relations.

We start with the following lemma which summarizes some important information of $\GR(\c)$ provided by the previous sections.

\begin{lemma}
Let $1$ denote the identity of $\GR(\c).$ Then we have
\begin{itemize}
\item[1)] $[V(1,0)]^n=[V(0,0)]=1.$
\item[2)] $[V(0,l)]=[V(0,1)][V(0,l-1)]-[V(1,0)][V(0,l-2)],$ \quad $\forall \ 2 \leq l\leq d-1$.
\item[3)] $[V(0,1)][V(0,d-1)]=[V(0,d-1)]+[V(1,0)][V(0,d-1)].$
\end{itemize}
\end{lemma}

\begin{proof}
\begin{itemize}
\item[1)] Follows from Lemma 3.1.
\item[2)] Direct consequence of (3.20).
\item[3)] Direct consequence of (3.24).
\end{itemize}
\end{proof}

To present the Green ring $\GR(\c),$ we need the generalized Fibonacci polynomials used in \cite{coz,lz}. Set $f_1(x,y)=1$ and $f_2(x,y)=y,$ then define recursively
\begin{equation}
f_i(x,y)=yf_{i-1}(x,y)-xf_{i-2}(x,y), \quad \forall i > 2.
\end{equation}

\begin{lemma} \cite[Lemma 3.11]{coz} For all $n \ge 2,$ we have
\[ f_n(x,y)= \sum_{i=0}^{[\frac{n-1}{2}]}(-1)^i {n-1-i \choose i} x^iy^{n-1-2i}. \]
\end{lemma}


Now we have the following expression of element $[V(i,l)]$.

\begin{lemma} In the ring $\GR(\c)$
$$[V(i,l)]=[V(1,0)]^i f_{l+1}([V(1,0)],[V(0,1)]), \quad \forall 0 \le i \le n-1, \ l \leq d-1.$$ In particular, $\GR(\c)$ is generated by $[V(1,0)]$ and $[V(0,1)].$
\end{lemma}

\begin{proof} Direct consequence of Lemma 4.1 2) and Lemma 4.2. \end{proof}

Finally we are ready to present the Green ring $\GR(\c).$

\begin{theorem}
Let $\Z[x,y]$ be the polynomial ring over $\Z$ with two variables $x$ and $y.$ By $J$ we denote the ideal of $\Z[x,y]$ generated by
$\{x^n-1,(y-x-1)f_{d}(x,y)\}.$ The Green ring $\GR(\c)$ is isomorphic to the quotient ring $\Z[x,y]/J.$
\end{theorem}
\begin{proof}Define a ring homomorphism by
\begin{eqnarray*}
F: \Z[x,y] & \longrightarrow & \GR(\c) \\
         x  & \mapsto & [V(1,0)], \\
         y  & \mapsto & [V(0,1)].
\end{eqnarray*}
By Lemma 4.3, $F$ is surjective. Moreover, by 1) and 3) of Lemma 4.1 and Lemma 4.3 with $i=0$ and $l=d-1,$ it is easy to see that $J \subset \ker F.$ So $F$ induces a ring epimorphism:
\begin{eqnarray*}
\overline{F}: \Z[x,y]/J & \longrightarrow & \GR(\c)\\
        \overline{x} & \mapsto & [V(1,0)], \\
         \overline{y}  & \mapsto & [V(0,1)].
\end{eqnarray*}
Notice that $\{\overline{x}^i\overline{y}^j\}_{0\leq i\leq n-1,
0\leq j\leq d-1}$ is a $\Z$-basis of $\Z[x,y]/J$ hence its $\Z$-dimension is $nd,$ which is equal to the $\Z$-dimension of $\GR(\c).$ That is, the map $\overline{F}$ is an isomorphism.
\end{proof}

We conclude the paper with the following remarks.

\begin{remark} Keep the same notations $n,d,s,q,Z_n$ and $\c(n,s,q)$ as in Section 2.
\begin{enumerate}
  \item If one prefers to work on a path algebra rather than a path coalgebra, then one can realize $\c(n,s,q)$ as the representation category of the quasi-Hopf algebra $H(n,s,q)$ which is dual to $M(n,s,q).$ As an algebra, $H(n,s,q)$ is isomorphic to $kZ_n^a/J^d$ where $kZ_n^a$ is the associated path algebra of the quiver $Z_n$ and $J$ is the ideal generated by the set of arrows of $Z_n.$
  \item When $s=0,$ then $d=\operatorname{ord} q$ is a factor of $n$ and $M(n,0,q)$ is an $nd$-dimensional Hopf algebra. The dual of $M(n,0,q)$ was considered in \cite{cc} by Cibils and the corresponding Clebsch-Gordan formulae were obtained there.
  \item When $s=0$ and $\operatorname{ord} q=n,$ the algebra $M(n,0,q)$ is the Taft Hopf algebra \cite{t} whose Green ring $r(H_n(q))$ (of representation category) was presented in \cite{coz}. It is well-known that the Taft algebra is self-dual, hence $\GR(\c(n,0,q)$ and $r(H_n(q))$ coincide.
  \item The Green rings $r(H_{n,d})$ of the generalized Taft algebras \cite{hcz} were presented in \cite{lz}. It turns out that the representation category of $H_{n,d}$ is a pointed tensor category of finite type, but not connected. The category can be presented as $n/d$ copies of $\c(d,0,q)$ with $d=\operatorname{ord} q,$ see \cite{hcz,qha3}. The dual of $H_{n,d}$ is exactly $M(n,0,q)$ with $d=\operatorname{ord} q,$ see \cite{chyz}. Thus, the Green ring $\GR(\c(n,0,q))$ is the Green ring of the \emph{corepresentation} category of $H_{n,d},$  or the Green ring of the  Hopf algebra $H_{n,d}^*$.  As the generalized Taft algebras $H_{n,d}$ are not self-dual if $n \ne d,$  so it is natural that $\GR(\c(n,0,q))$ and $r(H_{n,d})$ are not isomorphic.
\end{enumerate}

\end{remark}

\vskip 5pt

\noindent{\bf Acknowledgements:} The research of Huang was partially
supported by the SDNSF grant 2009ZRA01128 and the IIFSDU grant
2010TS021. The paper was written while Huang was visiting the
University of Antwerp and he is very grateful for its hospitality. He acknowledges the Belgian FWO for the financial support which makes the visit possible.


\begin{thebibliography}{99}

\bibitem{ass}
Assem, I.; Simson, D.; Skowronski, A.: \emph{Elements of the
Representation Theory of Associative Algebras. Vol. 1. Techniques of
Representation Theory}. London Mathematical Society Student Texts,
65. Cambridge University Press, Cambridge, 2006.


\bibitem{b}
Benson, D.J.: \emph{Representations and Cohomology: Volume 1, Basic
Representation Theory of Finite Groups and Associative Algebras}.
Cambridge Studies in Advanced Mathematics, vol. 30, Cambridge
University Press, Cambridge, 1991.




\bibitem{coz}
Chen, H.; Van Oystaeyen, F.; Zhang, Y.: \emph{The Green rings of Taft
algebras}. Proceedings of the AMS, to appear. arXiv:1111.1837v2
[math.RT].

\bibitem{chyz}
Chen, X.-W.; Huang, H.-L.; Ye, Y.; Zhang, P.: \emph{Monomial Hopf algebras}. J. Algebra 275 (2004), no. 1, 212-232.

\bibitem{cc}
Cibils, C.: \emph{A quiver quantum group}. Commun. Math. Phys.
157 (1993), 459-477.

\bibitem{cr}
Cibils, C.; Rosso, M.: \emph{Hopf quivers}. J. Algebra 254
(2002), no. 2, 241-251.


\bibitem{egno}
Etingof, P.; Gelaki, S.; Nikshych, D.; Ostrik, V.: \emph{Tensor categories}.
Lecture note for the MIT course 18.769, 2009. available at:
www-math.mit.edu/~etingof/tenscat.pdf

\bibitem{eo}
Etingof, P.; Ostrik, V.: \emph{Finite tensor categories}. Moscow Math.
J. 4 (2004) 627-654.

\bibitem{gr}
Green, J.A.: \emph{The modular representation algebra of a finite
group}. Illinois J. Math. 6(4) (1962), 607-619.

\bibitem{qha1}
Huang, H.-L.: \emph{Quiver approaches to quasi-Hopf algebras}. J.
Math. Phys. \textbf{50(4)} (2009) 043501, 9pp.

\bibitem{qha2}
Huang, H.-L.: \emph{From projective representations to quasi-quantum
groups}. Sci. China Math., 55 (2012) 2067-2080.

\bibitem{hcz}
Huang, H.-L.; Chen, H.; Zhang, P.: \emph{Generalized Taft algebras}. Algebra Colloq. 11 (2004), no. 3, 313¨C320.

\bibitem{qha3}
Huang, H.-L.; Liu, G.; Ye, Y.: \emph{Quivers, quasi-quantum groups and
finite tensor categories}. Comm. Math. Phys. 303 (2011) 595-612.

\bibitem{lz}
Li, L.; Zhang, Y.: \emph{The Green rings of the generalized Taft Hopf algebras}. To appear in Contemporary Math. AMS, Proceedings of Hopf Algebra and Tensor Categories conference in Almeria, 2011. arXiv:1210.4245 [math.RT].


\bibitem{t}
Taft, E.J.: \emph{The order of the antipode of finite dimensional Hopf algebras}. Proc. Natl. Acad. Sci. USA 68 (1971) 2631¨C2633.

\bibitem{w1}
Witherspoon, S.J.: \emph{The representation ring of the quantum
double of a finite group}. J. Algebra 179 (1996), no. 1, 305-329.

\bibitem{w2}
Witherspoon, S.J.: \emph{The representation ring of the twisted
quantum double of a finite group}. Canad. J. Math. 48 (1996), no. 6,
1324-1338.

\bibitem{w}
Woodcock, D.: \emph{Some categorical remarks on the representation theory
of coalgebras}. Comm. Algebra 25 (1997) 2775-2794.



\end{thebibliography}
\end{document}